\documentclass[a4paper,12pt]{amsart}
\usepackage{amssymb,amscd,amsmath,latexsym}
\usepackage[dvips]{graphicx}

\theoremstyle{plain}

\newtheorem{theo}{Theorem}[section]
\newtheorem{prop}[theo]{Proposition}

\theoremstyle{definition}

\newtheorem{prob}{Problem}

\theoremstyle{remark}
\newtheorem*{rema}{Remark}

\newtheorem{exam}[theo]{Example}

\numberwithin{equation}{section}

\newcommand{\field}[1]{\mathbb{#1}}
\newcommand{\C}{\field{C}}
\newcommand{\R}{\field{R}}
\newcommand{\Z}{\field{Z}}

\renewcommand{\P}{\mathcal P}

\DeclareMathOperator{\Hom}{Hom}
\DeclareMathOperator{\rank}{rank}

\DeclareMathOperator{\Aut}{Aut}
\DeclareMathOperator{\cf}{cf}

\def\S1{\C^*}

\def\uC{\underline{\C}}

\begin{document}

\title[Calssification problems of toric manifolds]
{Classification problems of toric manifolds via topology}
\date{\today}

\author{Mikiya Masuda}
\address{Department of Mathematics, Osaka City University, Sugimoto, 
Sumiyoshi-ku, Osaka 558-8585, Japan}
\email{masuda@sci.osaka-cu.ac.jp}
\author{Dong Youp Suh} 
\address{Department of Mathematics, Korea Advanced
Institute of Science and Technology, Gu-sung Dong, Yu-sung Gu,
Daejeon 305-701, Korea} 
\email{dysuh@math.kaist.ac.kr}
\thanks{The first author was partially supported by Grant-in-Aid for 
Scientific Research 4102-17540092
and the second author was partially supported by Korea Science and 
Engineering Foundation Grant  R11-2007-035-00000-0.}

\date{\today}

\dedicatory{This paper is dedicated to Professor Akio Hattori on his 
77th birthday.}


\begin{abstract}
We propose some problems on the classification of 
toric manifolds from the viewpoint of topology and survey related results. 
\end{abstract}

\maketitle

\section{Toric manifold and fan} \label{1}


A toric variety $X$ of dimension $n$ is a normal complex algebraic 
variety with an action of an $n$-dimensional algebraic torus $(\C^*)^n$ 
having a dense orbit.  
Let $X'$ be another toric variety of complex dimension $n'$ with an action 
of an $n'$-dimensional algebraic torus $(\C^*)^{n'}$. A map from $X$ 
to $X'$ is a morphism $f\colon X\to X'$ together with 
a homomorphism $\rho\colon (\C^*)^n\to (\C^*)^{n'}$ such that 
$f(tx)=\rho(t)f(x)$ for any $t\in (\C^*)^n$ and $x\in X$. 
Among toric varieties, 
compact smooth toric varieties, which we call {\em toric manifolds}, 
are well studied.  
In this article, we propose some problems on the classification of 
toric manifolds from the viewpoint of topology and survey related results. 

A \emph{rational convex polyhedral cone} in $\R^n$ is a cone 
spanned by a finitely many vectors in $\Z^n$, and it is called 
\emph{strong} if the origin is the apex. 
A fan in $\R^n$ is a non-empty collection $\Delta$ of rational 
strongly convex polyhedral cones in $\R^n$ satisfying the following conditions:
\begin{enumerate}
\item Each face of a cone in $\Delta$ is a also a cone in $\Delta$.
\item The intersection of two cones in $\Delta$ is a face of each.
\end{enumerate}
A fan $\Delta$ is called \emph{complete} if the union of 
cones in $\Delta$ covers the entire space $\R^n$, and \emph{non-singular} 
if every cone of dimension $k$ in $\Delta$ is spanned by $k$ integral 
vectors which form a part of a basis of $\Z^n$. 
Let $\Delta'$ be another fan in $\R^{n'}$.  A map from $\Delta$ to $\Delta'$ 
is a linear map from $\R^n$ to $\R^{n'}$ which maps $\Z^n$ into 
$\Z^{n'}$ and a cone in $\Delta$ into a cone in $\Delta'$.

A fundamental result in the theory of toric varieties says that 
the category of toric varieties is equivalent to the category of fans 
(see \cite[Theorem 1.5 and Theorem 1.13]{oda88}, also 
\cite[Theorem 4.1]{oda78}). 
A toric variety $X$ is compact if and only if 
the fan $\Delta_X$ of $X$ is complete, and 
smooth if and only if $\Delta_X$ is non-singular. 
Therefore, a toric variety $X$ is a toric manifold if and only if 
$\Delta_X$ is complete and non-singular. 

The complex projective space $\C P^n$ with a linear action of $(\C^*)^n$ 
is a toric manifold.  A product of finitely many toric manifolds 
is again a toric manifold with the product action, so a product of 
finitely many complex projective spaces is a toric manifold.  Here is 
a bit more non-trivial example of toric manifolds. 

\begin{exam} \label{bott}
Let $B$ be a toric manifold of complex dimension $k$.  Let $\gamma_i\to B$ 
$(i=1,\dots,\ell)$ be $(\C^*)^k$-equivariant line bundles over $B$. 
Each $\gamma_i$ has a $\C^*$-action defined by scalar multiplication 
so that the sum $\oplus_{i=1}^\ell\gamma_i$ has an action of 
$(\C^*)^{k+\ell}$.  
Let $\uC$ be a trivial line bundle over $B$ with fiber $\C$ on which 
the $(\C^*)^k$-action is trivial.  
Then the projectivization of $\uC\oplus_{i=1}^\ell \gamma_i$ 
has an induced action of $(\C^*)^{k+\ell}$ and is again a toric manifold. 

Starting with $B$ as a point and repeating the above construction, 
say $n$ times, we obtain a sequence of toric manifolds:
\begin{equation} \label{nbott}
B_n\stackrel{p_n}\longrightarrow B_{n-1} \stackrel{p_{n-1}}\longrightarrow
\dots \stackrel{p_2}\longrightarrow B_1 \stackrel{p_1}\longrightarrow 
B_0=\{\text{a point}\}
\end{equation}
where the fiber of $p_j\colon B_j\to B_{j-1}$ for $j=1,\dots,n$ is a 
complex projective space.  We call the above sequence 
(or often the top manifold $B_n$) an $n$-stage \emph{generalized Bott tower}, 
and especially call it an $n$-stage \emph{Bott tower} when each fiber is 
$\C P^1$.  Note that a Hirzebruch surface is a 2-stage Bott tower. 
The name of Bott tower was introduced and its study was initiated 
by Grossberg-Karshon 
\cite{gr-ka94}.  See \cite{civa05}, \cite{ci-ra05} and \cite{ma-pa06} for 
further study on Bott towers. 
\end{exam}

\section{Equivariant cohomology of a toric manifold} \label{2}

Equivariant cohomology fits very well to the study of toric manifolds, 
which we shall explain in this section.  We refer the reader to 
\cite{ha-ma03}, \cite{masu99} and \cite{mukh05} for details.

We set $T=(\C^*)^n$ and let $X$ be a toric manifold of complex dimension $n$ 
with an action of $T$. 
Associated with the universal principal 
$T$-bundle $ET\to BT$, we obtain a fibration 
\begin{equation} \label{fibra}
X\stackrel{\iota}\longrightarrow ET\times_T X\stackrel{\pi}
\longrightarrow BT
\end{equation}
where $ET\times_T X$ is the orbit space of $ET\times X$ by the diagonal 
$T$-action.
The equivariant cohomology of a toric manifold $X$ is 
the ordinary cohomology of the total space of the above fibration, that is, 
\[
H^*_T(X):=H^*(ET\times_T X)
\]

Let $X_i$ $(i=1,\dots,m)$ be invariant divisors of $X$. 
Since $X_i$ and $X$ are complex manifolds, they have canonical orientations. 
Let $\tau_i$ 
be the image of the unit element in $H^0_T(X_i)$ by the equivariant Gysin 
homomorphism from $H^0_T(X_i)$ to $H^2_T(X)$ induced by the inclusion map 
from $X_i$ to $X$.  We may think of $\tau_i$ as the Poincar\'e dual of 
the cycle $X_i$ in equivariant cohomology. 
The invariant divisors $X_i$ intersect transversally. Therefore, 
for each subset $I$ of $\{1,\dots,m\}$ the Poincar\'e dual of an intersection 
$\cap_{i\in I}X_i$ is a cup product $\prod_{i\in I}\tau_i$, so that 
the product $\prod_{i\in I}\tau_i$ vanishes if the intersection 
$\cap_{i\in I}X_i$ is empty.  It turns out that 
\begin{equation} \label{ring}
H^*_T(X)=\Z[\tau_1,\dots,\tau_m]/(\prod_{i\in I}\tau_i\mid \cap_{i\in I}X_i
=\emptyset).
\end{equation}
Since the underlying simplicial complex $\Sigma_X$ of the fan $\Delta_X$ of 
$X$ is given by 
\begin{equation} \label{Sigma}
\Sigma_X=\{ I\subset \{1,\dots,m\}\mid \cap_{i\in I}X_i\not=\emptyset\},
\end{equation}
the fact (\ref{ring}) shows that 
$H^*_T(X)$ is the face ring (or Stanley-Reisner ring) of 
the simplicial complex $\Sigma_X$, in particular determined 
by $\Sigma_X$.  Conversely, $H^*_T(X)$ as a graded 
ring determines the underlying 
simplicial 
complex, that is, if $H^*_T(X)\cong H^*_T(Y)$ as graded rings, 
then the underlying 
simplicial complexes $\Sigma_X$ and $\Sigma_Y$ are isomorphic 
(see \cite{br-gu96} or \cite[Problem 31 in p.141]{stan96}). 

We note that $H^*_T(X)$ is not only a graded ring but also 
a graded algebra over $H^*(BT)$ through $\pi^*\colon H^*(BT)\to H^*_T(X)$ 
where $\pi$ is the projection in (\ref{fibra}).  
Since $H^*(BT)$ is a polynomial ring 
generated by $H^2(BT)$ and $H^2_T(X)$ is additively 
generated by $\tau_i$'s, the algebra structure can be detected 
once we know how $\pi^*(u)$ is described as a linear 
combination of $\tau_i$'s for $u\in H^2(BT)$. 
The coefficient of $\tau_i$ in the linear expression of $\pi^*(u)$ 
is a linear function of $u$, so that there is a unique element 
$v_i\in H_2(BT)$ for each $i$ such that 
\begin{equation} \label{algeb}
\pi^*(u)=\sum_{i=1}^m\langle u,v_i\rangle \tau_i
\end{equation}
where $\langle\ ,\ \rangle$ denotes the natural pairing between cohomology 
and homology.  Therefore, the simplicial complex $\Sigma_X$ together with 
the elements $v_i$'s determines the algebra structure of $H^*_T(X)$ over 
$H^*(BT)$.  

Since $T=(\C^*)^n$, there are natural identifications
\[
\Z^n= H_2(BT)=\Hom(\C^*,T)
\]
where the last one denotes the group of homomorphism from $\C^*$ to $T$. 
We denote by 
$\lambda_v$ the element in $\Hom(\C^*,T)$ corresponding to $v\in H_2(BT)$. 
It turns out that $\lambda_{v_i}(\C^*)$ is the $\C^*$-subgroup of $T$ 
which fixes the invariant divisor $X_i$ pointwise. 
For each member $I$ of $\Sigma_X$ we form a cone in $H_2(BT)\otimes_\Z\R
=\R^n$ spanned by vectors $v_i$'s for $i\in I$.  The collection of 
those cones (and the cone consisting of only the origin) agrees with the 
fan $\Delta_X$ of $X$. Therefore, the data of $\Sigma_X$ 
together with the set of vectors $\{v_i\}$ is equivalent to the data of the 
fan $\Delta_X$. 

The restriction map $\iota^*\colon H^*_T(X)\to H^*(X)$ is surjective 
and its kernel is generated by $\pi^*(H^2(BT))$. 
We set $\mu_i=\iota^*(\tau_i)\in H^2(X)$. It is the (ordinary) Poincar\'e 
dual of the cycle $X_i$ in $X$.  
Then we obtain the following well-known result 
from (\ref{ring}), (\ref{Sigma}) and (\ref{algeb}):
\begin{equation} \label{tcoho}
H^*(X)=\Z[\mu_1,\dots,\mu_m]/\mathcal I
\end{equation}
where $\mathcal I$ is the ideal generated by the following two types 
of elements:
\begin{enumerate}
\item[(1)] $\prod_{i\in I}\mu_i$ \quad for $I\notin\Sigma_X$,
\item[(2)] $\sum_{i=1}^m\langle u,v_i\rangle\mu_i$ \quad for $u\in H^2(BT)$,
\end{enumerate}
see \cite[p.106]{fult93} and \cite[p.134]{oda88}.

\section{Classification of toric manifolds as varieties} \label{3}

Before we discuss topological classification of toric manifolds, we shall 
recall some known results on the classification of toric manifolds as 
varieties. The following is fundamental. 

\begin{prop}
For toric manifolds $X$ and $Y$ of complex dimension $n$, 
the following are equivalent. 
\begin{enumerate}
\item[(V1)] $X$ and $Y$ are \emph{non-equivariantly} isomorphic as varieties.
\item[(V2)] $X$ and $Y$ are \emph{weakly equivariantly} 
isomorphic as varieties, i.e., 
there is an isomorphism $f\colon X\to Y$ together with 
an automorphism $\rho$ of $T=(\C^*)^n$ such that $f(tx)=\rho(t)f(x)$ 
for any $t\in T$ and $x\in X$. 
\item[(V3)] The fans $\Delta_X$ and $\Delta_{Y}$ of $X$ and $Y$ are 
\emph{isomorphic (or unimodularly equivalent)}, i.e., there is a unimodular 
automorphism of $\R^n$ which maps cones in 
$\Delta_X$ to cones in $\Delta_{Y}$ bijectively.
\end{enumerate}
\end{prop}

\begin{proof}
The implication from (V2) to (V1) is trivial, and 
the equivalence of (V2) and (V3) follows from the fundamental result 
in the theory of toric varieties mentioned in Section~\ref{1}. 
So, it suffices to prove the implication from (V1) to (V2).  
Suppose that there is an isomorphism $f\colon X\to Y$. 
Then it induces a group isomorphism $f_*\colon \Aut(X)\to \Aut(Y)$ 
between the automorphism groups of $X$ and $Y$.  In fact, $f_*(g)$ for 
$g\in\Aut(X)$ is given by $fgf^{-1}$. It is known that 
the automorphism group of a toric manifold is a linear algebraic group 
with the acting torus as a maximal algebraic torus  
(\cite[Section 3.4]{oda88}) and that 
maximal algebraic tori in a linear algebraic group are all 
conjugate (\cite[Corollary A in p.135]{hump87}). 
Let $T_X$ (resp. $T_{Y}$) be the maximal torus of $\Aut(X)$ 
(resp. $\Aut(Y)$) determined by the torus acting on $X$ (resp. $Y$). 
Since $f_*$ is an isomorphism, $f_*(T_X)$ is a maximal algebraic torus of 
$\Aut(Y)$, so that there is 
an element $h\in \Aut(Y)$ such that $f_*(T_X)=h^{-1}T_{Y}h$.  Then 
the composition 
$hf\colon X\to Y$ induces an isomorphism $(hf)_*\colon \Aut(X)\to 
\Aut(Y)$ mapping $T_X$ to $T_{Y}$. This implies that the isomorphism 
$hf$ is weakly equivariant.  
\end{proof}

Because of the equivalence between (V1) and (V3) above, 
the classification problem of toric manifolds as varieties 
reduces to the combinatorial problem of classifying fans up to isomorphism. 
Based on this fact, the classification of toric manifolds of dimension $n$ 
as varieties has been completed in several cases.  For instance,
\begin{enumerate}
\item[(1)] $n=2$ (\cite[Theorem 1.28]{oda88}), or 
$n=3$ and the 2nd Betti number (or Picard number) is five 
or less (\cite[Theorem 1.34]{oda88}). 
\item[(2)] Smooth toric Fano varieties (i.e. toric manifolds with an ample 
anticanonical divisor) of dimension $n\le 4$ (\cite[Proposition 2.21]
{oda88} for $n=2$, \cite[p.90]{oda88} for $n=3$, 
\cite{baty99} and \cite{sato00} for $n=4$). 
\item[(3)] 2-stage generalized Bott towers (\cite{klei88}). 
\end{enumerate}
See \cite{ewal96}, \cite{sato06} and their 
references for further classification results. 

\medskip

If two toric manifolds $X$ and $Y$ are (weakly equivariantly) isomorphic as 
varieties, then their equivariant cohomology algebras are \emph{weakly 
isomorphic}, i.e., there is a graded ring isomorphism $\Phi\colon H^*_{T}(Y)
\to H^*_T(X)$ together with an automorphism $\rho$ of $T$ 
such that $\Phi(u\omega)=\rho^*(u)\Phi(\omega)$ for any $u\in H^*(BT)$ 
and $\omega\in H^*_T(Y)$, where $\rho^*$ is the automorphism of $H^*(BT)$ 
induced by $\rho$. It turns out that the converse holds (\cite{masu07}), so 
we have the following another equivalent statement to (V1) above:
\begin{enumerate}
\item[(V4)] 
$H^*_T(X)$ and $H^*_{T}(Y)$ are weakly isomorphic as algebras. 
\end{enumerate}

\section{Topological classification of toric manifolds} \label{4}

We shall consider the topological classification of toric manifolds. 
If two toric manifolds are isomorphic as varieties, then they are 
homeomorphic, but the converse does not hold in general.  
Here is a well-known simple example. 

\begin{exam}\label{hirze}
For an integer $a$, we denote by $\gamma^a$ the $a$ fold tensor product 
of the canonical line bundle $\gamma$ over $\C P^1$.  
Let $\underline{\C}$ be the trivial line bundle over $\C P^1$. 
Then $P(\underline{\C}\oplus\gamma^a)$ is a Hirzebruch surface 
(i.e., 2-stage Bott tower). It is well known that 
$P(\underline{\C}\oplus\gamma^a)$ and 
$P(\underline{\C}\oplus\gamma^b)$ are isomorphic 
as varieties if and only if $|a|=|b|$ (see \cite[Theorem 1.28 (3)]{oda88}), 
while they are homeomorphic if and only if $a\equiv b\pmod 2$.  


Here is a proof of the homeomorphism classification above. 
The \lq\lq only if" part follows from the fact that 
$P(\underline{\C}\oplus\gamma^a)$ is spin if and only if $a$ is even. 
One can also check that if
$H^*(P(\underline{\C}\oplus\gamma^a))$ is isomorphic to 
$H^*(P(\underline{\C}\oplus\gamma^{b}))$ as graded rings, then $a\equiv b\pmod 2$.
The proof of the \lq\lq if" part is as follows. 
Note that $P(E)$ is homeomorphic to $P(E\otimes\eta)$ for any complex 
vector bundle $E$ and any complex line
bundle $\eta$. Suppose $a\equiv b \pmod 2$. Then $b-a=2\ell$ for
some $\ell\in \Z$ and we have homeomorphisms
\[
P(\underline{\C}\oplus\gamma^a)\cong
P((\underline{\C}\oplus\gamma^a)\otimes\gamma^{\ell})=
P(\gamma^\ell\oplus\gamma^{a+\ell}).
\]
Since both $\gamma^\ell\oplus\gamma^{a+\ell}$ and
$\underline{\C}\oplus\gamma^{b}$ are over $\C P^1$ and have the same 
first Chern class, they are isomorphic. Hence the last space
above is $P(\underline{\C}\oplus\gamma^{b})$, proving 
the \lq\lq if" part. 

In fact, $P(\underline{\C}\oplus\gamma^a)$ is homeomorphic to 
$\C P^1\times \C P^1$ (resp. $\C P^2\#\overline{\C P^2})$) when $a$ is even 
(resp. odd), where $\overline{\C P^2}$ denotes $\C P^2$ with reversed 
orientation. 
\end{exam}

As remarked at the end of Section~\ref{3}, equivariant cohomology 
determines the isomorphism type of toric manifolds as varieties. 
This leads us to ask how much information \emph{ordinary} cohomology has 
for toric manifolds, and the example above shows that ordinary cohomology ring 
distinguishes the homeomorphism types of Hirzebruch surfaces. We ask 

\begin{prob} \label{cohom}
Are toric manifolds $X$ and $Y$ homeomorphic 
if $H^*(X)\cong H^*(Y)$ as graded rings (or if $X$ and $Y$ are homotopy equivalent)? 
\end{prob}

There are 
infinitely many closed smooth manifolds which are homotopy equivalent to 
$\C P^n$ but not homeomorphic to each other for $n\ge 3$ (\cite{hsia66}, 
\cite{mo-ya66}). 
More generally, surgery theory would imply a similar result for 
many toric manifolds different from $\C P^n$. 
So, Problem~\ref{cohom} might be bold but we have a feeling that most of 
manifolds do not have large symmetry and 
we do not know any counterexample to Problem~\ref{cohom}. 

We shall give some evidence supporting Problem~\ref{cohom}. 

\begin{prop} 
Problem~\ref{cohom} has an affirmative solution 
for toric manifolds of complex dimension one and two.
\end{prop}

\begin{proof}
A toric manifold of complex dimension one is $\C P^1$, so the proposition 
is trivial in dimension one. 
Simply connected real 4-dimensional closed smooth
manifolds are classified up to homeomorphism by isomorphism classes of 
the bilinear forms 
defined by the intersection paring of real 2-cyclyes (\cite{free82}), 
so the homeomorphism types of those manifolds are distinguished by 
their cohomology rings. 
This together with the fact that any toric manifold is smooth and 
simply connected 
(\cite[Section 3.2]{fult93}) implies the proposition in dimension two. 
\end{proof}

\begin{rema}
Any toric manifold of complex dimension two is obtained by blowing up 
$\C P^2$ or a Hirzebruch surface finitely many times 
(\cite[Theorem 1.28]{oda88}).  
As remarked before, any Hirzebruch surface 
is homeomorphic to either 
$\C P^1\times \C P^1$ or $\C P^2\#\overline{\C P^2}$.  
Although $\C P^1\times \C P^1$ and $\C P^2\#\overline{\C P^2}$ are not 
homeomorphic, they become homeomorphic 
after blowing up (in 
other words, after taking the connected sum with $\overline{\C P^2}$). 
Thus, a toric manifold of complex dimension two is homeomorphic 
(even diffeomorphic) to $\C P^2$, $\C P^1\times \C P^1$ or the connected 
sum of $\C P^2$ with a finite number of copies of $\overline{\C P^2}$ 
(\cite{fi-ya94}). 
\end{rema}

Besides the proposition above, 
there are some partial affirmative solutions to 
Problem~\ref{cohom}.  For instance, 
\begin{enumerate}
\item[(1)] $X=(\C P^1)^n$ and $Y$ is an arbitrary toric 
manifold (\cite{ma-pa07}). 
\item[(2)] $X$ is a product of complex projective spaces and $Y$ is 
an arbitrary generalized Bott tower (\cite{ch-ma-su07}). 
\item[(3)] $X$ and $Y$ are both 2-stage generalized Bott towers 
(\cite{ch-ma-su07}). 
\end{enumerate}
The reader can find more partial affirmative solutions to 
Problem~\ref{cohom} in \cite{ch-ma-su07}. 

\medskip

The simplicial complex $\Sigma_X$ associated with a toric manifold $X$ is 
determined by the equivariant cohomology of $X$ as explained in 
Section~\ref{2}. As for ordinary cohomology, 
the number $f_i$ of $i$-simplices in $\Sigma_X$ is determined by $H^*(X)$. 
In fact, since 
$H^*_T(X)$ is the face ring of the simplicial complex $\Sigma_X$, we have 
\[
\sum_{i=0}^\infty \rank H^{2i}_T(X)t^i=
1+\sum_{i=0}^{n-1}\frac{f_it^{i+1}}{(1-t)^{i+1}}
\]
(see \cite[Theorem 1.4 in p.54]{stan96}) 
while since $H^*_T(X)$ is isomorphic to $H^*(BT)\otimes H^*(X)$ as graded 
modules, the left hand side above is equal to 
\[
\frac{1}{(1-t)^n}\sum_{k=0}^n\rank H^{2k}(X)t^k. 
\]
Equating the above two and replacing $t$ by $(s+1)^{-1}$, we see that 
$f_i$ agrees with the coefficient of $s^{n-i-1}$ in 
$\sum_{k=0}^n\rank H^{2k}(X)(s+1)^{n-k}$. 

Although $H^*(X)$ contains some information on $\Sigma_X$ as observed above, 
it is not true in general that $\Sigma_X$ is determined by $H^*(X)$ 
as is seen in the following example. 

\begin{exam} \label{stack}
We use the fact that a maximal simplex in $\Sigma_X$ corresponds to a fixed 
point in $X$ and blowing up $X$ at a fixed point equivariantly corresponds to 
applying a stellar subdivision (\cite[p.70]{ewal96}) 
to the simplex corresponding to the fixed point (\cite[Section 2.6]{fult93}). 

We start with $\C P^2\times \C P^1$ with a standard action of $(\C^*)^3$. 
The simplicial complex associated with it 
is the suspension of the boundary complex of 
a triangle. Let $Y$ be a toric manifold obtained by blowing up 
$\C P^2\times\C P^1$ at a fixed point equivariantly.  Although there are six 
fixed points in $\C P^2\times \C P^1$, the 
simplicial complex associated with $Y$ does not depend on the fixed point 
chosen for blowing up. 
Then we blow up $Y$ at a fixed point equivariantly. 
In this case the simplicial complex associated with the resulting toric 
manifold $X$ does depend on the fixed point chosen for blowing up.  In fact, 
we obtain three 
different underlying simplicial complexes shown as the first three in the 
second line in p.192 of \cite{oda88}. However, $X$ is 
homeomorphic to the connected sum of $\C P^2\times \C P^1$ with two copies 
of $\C P^3$ (with reversed orientation) regardless of the chosen fixed point.  
Therefore this gives a desired example. 
\end{exam}

On the contrary, $\Sigma_X$ is sometimes determined by $H^*(X)$. 
For instance, this is the case when $\Sigma_X$ is the boundary complex of 
a crosspolytope (\cite{ma-pa07}).  Motivated by this, 
we say that 
the simplicial complex $\Sigma_X$ associated with a toric manifold $X$ 
is \emph{rigid} if $\Sigma_X\cong \Sigma_Y$ whenever $H^*(X)\cong H^*(Y)$ 
as graded rings.
The boundary complex $\partial \Delta^n$ of a simplex $\Delta^n$ of dimension 
$n$ is rigid because a toric manifold with $\partial \Delta^n$ as 
the associated simplicial complex is only $\C P^n$.  Moreover, 
the result mentioned above asserts that 
a join $\partial \Delta^0*\dots*\partial\Delta^0$, which is isomorphic to 
the boundary complex of a crosspolytope, is rigid.  We ask 

\begin{prob} \label{cosim}
Which simplicial complex is rigid or not rigid?  
In particular, is a join $\partial\Delta^{n_1}*\dots*\partial\Delta^{n_k}$ 
rigid for any value of $n_i$'s and any $k$?
\end{prob}

As an intermediate step to Problem~\ref{cohom}, we may ask 

\begin{prob} \label{rcohom}
Are toric manifolds $X$ and $Y$ homeomorphic if $H^*(X)\cong H^*(Y)$ as graded rings 
and $\Sigma_X$ is rigid, or more generally, if $H^*(X)\cong H^*(Y)$ as graded rings 
and $\Sigma_X\cong \Sigma_Y$? 
\end{prob}

Although Problems~\ref{cohom} and~\ref{rcohom} are stated 
in the topological category, affirmative results known so far to 
those problems actually hold in the smooth category. Quite generally, 
we may ask 

\begin{prob} \label{homeo}
Are two toric manifolds diffeomorphic if they are homeomorphic?
\end{prob}

\section{Pontrjagin class of a toric manifold} \label{5}

A homeomorphism between closed manifolds preserves their rational 
Pontrjagin classes as is well known.  
Since the cohomology group of a toric manifold has 
no torsion, any homeomorphism between toric manifolds preserves their 
integral Pontrjagin classes. 
Therefore, as a step toward Problem~\ref{cohom} we may ask 

\begin{prob} \label{pontj}
If two toric manifolds have isomorphic cohomology rings, then 
is there an isomorphism between their cohomology rings which 
preserves their Pontrjagin classes? 
\end{prob}

We have an explicit description (\ref{tcoho}) of $H^*(X)$ for a toric 
manifold $X$, and it is known 
that the Chern class of $X$ is given by $\prod_{i=1}^m(1+\mu_i)$ 
(see \cite[Theorem 3.12 in p.131]{oda88}), 
so the Pontrjagin class $p(X)$ of $X$ is given by 
\begin{equation} \label{tpont}
p(X)=\prod_{i=1}^m(1+\mu_i^2).
\end{equation}
Therefore Problem~\ref{pontj} is purely algebraic. 

The affirmative solution to Problem~\ref{cohom} implies the affirmative 
solution to Problem~\ref{pontj} as remarked at the beginning of this section, 
but the results of \cite{wall66} and \cite{jupp73} show that the 
converse implication also holds in complex dimension three. 

Problem~\ref{cohom} asks whether there is a homeomorphism between 
toric manifolds if there is an isomorphism between their cohomology rings. 
More strongly we may ask 

\begin{prob} \label{surje}
Is any isomorphism between cohomology rings of toric manifolds induced 
by a homeomorphism between the manifolds?
In particular, does any isomorphism between cohomology rings of toric manifolds 
preserve the Pontrjagin classes of the manifolds? 
\end{prob}

Some partial affirmative solution to the latter part of Problem~\ref{surje} 
could be found in \cite{ch-ma-su07}. 

Problem~\ref{pontj} or the latter part of Problem~\ref{surje} reminds us 
of a conjecture by Petrie~\cite{petr72}, which 
says that if $M$ is a closed smooth manifold homotopy equivalent to $\C P^n$ 
and 
$M$ admits a non-trivial smooth action of $S^1$, then any homotopy equivalence 
between $M$ and $\C P^n$ preserves their Pontrjagin classes, i.e., 
$p(M)=(1+x^2)^{n+1}$ for a generator $x\in H^2(M)$. No counterexample is 
known and there are some partial affirmative solutions to the conjecture.  
Among them, it is proved 
in \cite{petr73} that $p(M)$ is of the above form if $M$ supports an effective 
smooth action of $(S^1)^n$.  See \cite{dess04} for related results.

\section{Quasitoric manifolds} \label{6}

The theory of toric manifolds can be developed in the topological category 
to some extent.  The pioneering work in this direction was done by 
Davis-Januszkiewicz in \cite{da-ja91}.  They introduced the notion of 
what is now called a \emph{quasitoric manifold} as a topological 
counterpart to a toric manifold in algebraic geometry, and showed that 
an analogous theory can be developed for quasitoric manifolds in 
the topological category. 
We refer the reader to a book \cite{bu-pa02} 
by Buchstaber-Panov for further development. 

A quasitoric manifold is a closed smooth manifold $M$ of real dimension $2n$ 
with a smooth action of $(S^1)^n$ such that 
\begin{enumerate}
\item[(1)] the action is \emph{locally standard} 
(that is, the action is locally same as a faithful real $2n$-dimensional 
representation of $(S^1)^n$ in the smooth category), and 
\item[(2)] the orbit space 
$M/(S^1)^n$ is a simple convex polytope. 
\end{enumerate}
The restricted action of $(S^1)^n$ on a toric manifold $X$ is locally standard 
and the orbit space $X/(S^1)^n$ is a manifold with corners whose faces 
(even $X/(S^1)^n$ itself) are all contractible and any multiple intersection 
of faces is connected whenever it is non-empty.  When $X$ is projective, 
there is a moment map whose image identifies $X/(S^1)^n$ 
with a simple convex polytope. Therefore, a projective toric manifold 
provides an example of a quasitoric manifold. Even if a toric manifold 
is not projective, it often provides an example of a quasitoric manifold. 
For example, there are non-projective toric manifolds of complex dimension 
three (see \cite[Section 2.3]{oda88}), but any toric manifold of complex 
dimension three with the restricted action of $(S^1)^3$ 
is a quasitoric manifold, which follows from a famous theorem 
of Steinitz (see \cite[Theorem 4.1]{zieg95}). 
However, a toric manifold may fail to be a quasitoric manifold in higher 
dimensions. 


On the other hand, it is easy to find a quasitoric but not 
toric manifold.  For instance, $\C P^2\#\C P^2$ is a quasitoric 
manifold with an appropriate action of $(S^1)^2$ but not a toric 
manifold because it does not allow a complex structure (even an almost 
complex structure).   See \cite[Section 5]{masu99} for examples of 
quasitoric manifolds which are not toric but allow an almost complex 
structure invariant under the torus action. 

Let $M$ be a quasitoric manifold of dimension $2n$ 
with a simple convex polytope $P$ of dimension $n$ as the orbit space and 
let 
$$q\colon M\to P=M/(S^1)^n$$ 
be the quotient map. 
Let $M_i$ $(i=1,\dots,m)$ be a closed smooth codimension two submanifold of 
$M$ fixed pointwise under some $S^1$-subgroup of $(S^1)^n$. We call 
$M_i$'s \emph{characteristic submanifolds} of $M$. 
When $M$ is a toric manifold, $M_i$'s are invariant divisors. 
Note that $q(M_i)$ is a \emph{facet} (i.e. codimension one face) of $P$ and 
the map $q$ gives a one-to-one correspondence between the characteristic 
submanifolds of $M$ and the facets of $P$. 

The group $\Hom(S^1,(S^1)^n)$ of homomorphisms from $S^1$ to 
$(S^1)^n$ can 
naturally be identified with $\Z^n$, and we denote by $\lambda_v$ the 
element in $\Hom(S^1,(S^1)^n)$ corresponding to $v\in \Z^n$. 
Let $v_i$ be a primitive element in $\Z^n$ such that $\lambda_{v_i}(S^1)$ 
fixes $M_i$ pointwise.  Note that there are two choices of $v_i$ and 
the other one is $-v_i$.  We need an orientation data (called an 
\emph{omniorientation} in \cite{bu-pa02}) on $M$ and $M_i$ 
to make the choice of $v_i$ unique. 
When $M$ is a toric manifold, both $M$ and $M_i$ are complex manifolds and 
have canonical orientations, so that $v_i$'s are uniquely determined and 
$v_i$ agrees with the edge vector corresponding to $M_i$ in the fan. 

Let $P_i$ $(i=1,\dots,m)$ be the facets of $P$ 
such that $P_i=q(M_i)$.  The vectors $v_i$'s are assembled to 
define what is called the \emph{characteristic function} of $M$: 
\[
\Lambda_M \colon \{P_1,\dots,P_m\}\to \Z^n. 
\]
When $M$ is a toric manifold with $P$ as the orbit space by the restricted 
action of $(S^1)^n$, the simplicial complex $\Sigma_M$ associated with $M$ 
agrees with 
the boundary complex of the dual of $P$ and $v_i$'s are the edge 
vectors of the fan of $M$; so the characteristic function $\Lambda_M$ 
together with (the combinatorial type of) $P$ has an equivalent data to the 
fan of $M$. 

The map $\Lambda_M$ above has the property that 
whenever $n$ facets of $P$ meet at a vertex of $P$, 
their images by $\Lambda_M$ form a basis of $\Z^n$. A map 
from the set $\{P_1,\dots,P_m\}$ to $\Z^n$ 
possessing this property is called a \emph{characteristic function} on $P$. 
It is known that any characteristic function on $P$ can be realized as 
the characteristic function of some quasitoric manifold over $P$ 
(see \cite[Section 3]{bu-pa-ra06}\footnote{Such a quasitoric manifold is 
constructed in \cite[Section 1.5]{da-ja91} but it 
is not obvious how to give a smooth structure on the manifold.}). 

Let $M$ and $M'$ be quasitoric manifolds over $P$.  Then 
$\Lambda_M=\Lambda_{M'}$ if and only if there is an equivariant 
homeomorphism between $M$ and $M'$ covering the identity on $P$ 
(\cite{da-ja91}).  However an equivariant homeomorphism between them 
does not necessarily cover the identity on $P$. In general, it 
covers a self-homeomorphism of $P$ preserving the face structure of $P$. 
The group $\Aut(P)$ of self-homeomorphisms of $P$ preserving the face 
structure acts on the set $\cf(P)$ of characteristic functions on $P$ 
through 
the natural action on the set $\{P_1,\dots,P_m\}$.  One sees that 
$M$ and $M'$ are equivariantly homeomorphic if and only if 
$\Lambda_M$ and $\Lambda_{M'}$ are in the same orbit of $\Aut(P)$ in $\cf(P)$. 

Dobrinskaya \cite{dobr01} discusses the classification of characteristic 
functions over a given polytope $P$. In particular, she gives a criterion 
of when 
a quasitoric manifold over a product of simplices is a toric manifold 
(that is, a generalized Bott tower in this case), 
see also \cite{ch-ma-su07-2} and \cite{ma-pa07}. 

It is proved in \cite{da-ja91} that 
(\ref{tcoho}) and (\ref{tpont}) hold even for any quasitoric manifold, so 
one may ask the problems in Sections~\ref{4} and~\ref{5} 
for quasitoric manifolds.  
A partial affirmative solution to Problem 1 for quasitoric manifolds is given 
in \cite{ch-ma-su07-2} and \cite{ma-pa07}.

\section{Torus manifolds} \label{7}

The family of quasitoric manifolds 
may not contain the family of toric manifolds entirely but 
an analogous theory to the toric theory can be developed for quasitoric 
manifolds.  This implies that the theory can further be extended to a 
certain family of manifolds containing both toric manifolds and quasitoric 
manifolds. 

A \emph{torus manifold} $M$ introduced in \cite{ha-ma03} is a closed smooth 
orientable manifold of 
dimension $2n$ with a smooth effective action of $(S^1)^n$ having 
a fixed point. Obviously both toric manifolds and quasitoric manifolds are 
torus manifolds.  A simple example of a torus manifold which is neither 
toric nor quasitoric is as follows. 

\begin{exam}\label{S2n}
Let $S^{2n}$ be the $2n$-dimensional sphere  
identified with the following subset of $\C^{n}\times\R$:
$$
  \bigl\{ (z_1,\dots,z_n,y)\in\C^n\times\R\mid |z_1|^2+\dots+|z_n|^2+y^2=1 
  \bigr\},
$$
and define an action of $(S^1)^n$ on $S^{2n}$ by
$$
  (t_1,\dots,t_n)\cdot(z_1,\dots,z_n,y)=(t_1z_1,\dots,t_nz_n,y).
$$
A map
\[
(z_1,\dots,z_n,y)\to (|z_1|,\dots,|z_n|,y)
\]
induces a homeomorphism from the orbit space 
$S^{2n}/(S^1)^n$ onto the following subset of the $n$-sphere: 
\[
\{ (x_1,\dots,x_n,y)\in \R^{n+1}\mid x_1^2+\dots+x_n^2+y^2=1,\ x_1\ge 0,
\dots,x_n\ge 0\}.
\]

The orbit space $S^{2n}/(S^1)^n$ is a manifold with corners and 
every face of the orbit space (even the orbit space itself) is 
contractible.  The facets are images of real codimension 
two submanifolds $\{z_i=0\}$ of $S^{2n}$ $(i=1,\dots,n)$ 
under the quotient map above and the intersection of the $n$ facets consists 
of two points $(0,\dots,0,\pm 1)$ when $n\ge 2$. 
Therefore $S^{2n}$ with the above action is a torus manifold which is neither 
toric nor quasitoric when $n\ge 2$. 
\end{exam}

A simplicial poset $\P$ 
is a finite poset with a smallest element $\hat 0$ such that every 
interval $[\hat 0,y]$ for $y\in \P$, that is a subposet of $\P$ consisting of 
all elements between $\hat 0$ and $y$,  
is isomorphic to the set of all subsets of a finite set, ordered by 
inclusion.  
The set of all faces of a (finite) simplicial complex 
with empty set added forms a simplicial poset ordered by inclusion, 
where the empty set is the smallest element.  
Such a simplicial poset is called the \emph{face poset} of a 
simplicial complex, and two simplicial complexes are isomorphic if and 
only if their face posets are isomorphic.  Therefore, a simplicial 
poset can be thought of as a generalization of a simplicial complex. 

Although a simplicial poset $\P$ is not necessarily the face poset of 
a simplicial complex, it is always the face poset of a finite CW complex 
$\Gamma(\P)$. In fact, to each $y\in \P\backslash\{\hat 0\}$, 
we assign a (geometrical) simplex whose face poset is $[\hat 0,y]$ and 
glue those geometrical simplices according to the order relation in $\P$.
Then we get the CW complex $\Gamma(\P)$ such that all the cells 
are simplices and all the attaching maps are inclusions.  
A finite CW complex like $\Gamma(\P)$ is called 
a \emph{simplicial cell complex}.  We may say that 
the notion of simplicial poset is equivalent to that of simplicial cell 
complex. 

The face poset of the orbit space $S^{2n}/(S^1)^n$ in Example~\ref{S2n}
(with reversed order by inclusion) 
is not the face poset of a simplicial complex when $n\ge 2$.  
But it is the face poset of a simplicial cell complex 
formed from two $(n-1)$-simplices by gluing their boundaries via 
the identity map.  It can be thought of as the \emph{dual} of the boundary of 
the orbit space $S^{2n}/(S^1)^n$. 

The orbit space of a toric or quasitoric manifold is topologically trivial. 
This is also the case for the torus manifold in Example~\ref{S2n} but 
not always the case for an arbitrary torus manifold as the following example 
shows. 

\begin{exam} \label{MN}
%
Take a torus manifold $M$ of dimension $2n$ and a closed smooth manifold 
manifold $N$ of dimension $n$, and 
consider the free action of $(S^1)^n$ on the product 
$N\times (S^1)^n$ given by multiplication on the second factor. 
We choose a free orbit for each of $M$ and $N\times (S^1)^n$, remove 
their invariant open tubular neighborhoods and glue the resulting manifolds 
along their boundaries to get a new torus manifold $M'$.  
The orbit space $M'/(S^1)^n$ is the connected sum of $M/(S^1)^n$ 
and $N$ at interior points. Since $N$ can be arbitrary, 
the orbit space of a torus manifold 
is not necessarily topologically trivial unlike toric or quasitoric manifolds. 
\end{exam}

For a torus manifold $M$, characteristic submanifolds $M_i$'s 
can be defined similarly to the quasitoric case and they play the role 
of the invariant divisors $X_i$ for a toric manifold $X$. 
Similarly to the toric case, we get a simplicial complex $\Sigma_M$ 
and vectors $v_i$'s from the characteristic submanifolds $M_i$'s. 
Using these data, one can associate with $M$ 
a combinatorial object $\Delta_M$ called the \emph{multi-fan} of $M$ 
in a similar fashion to the toric case (\cite{ha-ma03}, \cite{masu99}). 
Precisely speaking, we assign orientations on $M$ and $M_i$'s 
(i.e.~an omniorientation) which make the choice of $v_i$'s unique 
as remarked in Section~\ref{6} and 
moreover we attach an integer to each cone of maximum dimension $n$. 
Such a cone corresponds to $n$ characteristic submanifolds
in $M$ and the integer attached to the cone counts the number of points 
(with sign determined by the omniorientation) 
in the intersection of the $n$ characteristic submanifolds. 
When $M$ is a toric manifold, the attached integers are all one 
(so that we may neglect them) and the multi-fan $\Delta_M$ is an ordinary fan, 
but unless $M$ is a toric manifold, the attached integers are not 
necessarily one 
and cones in $\Delta_M$ may overlap. 
Although $\Delta_M$ contains 
a lot of geometrical information on $M$, it does not determine $M$ in general. 
For instance, the torus manifolds $M$ and $M'$ in Example~\ref{MN}
are not equivariantly homeomorphic in general, but their mutli-fans are 
same because a mulit-fan is defined using only characteristic submanifolds and 
the characteristic submanifolds of $M$ and $M'$ are same. 

If the action on a torus manifold of dimension $2n$ is locally standard, 
then its orbit space is a compact \emph{nice} manifold of dimension $n$ 
with corners, where \lq\lq 
nice" means that there are exactly $n$ codimension-one faces meeting at 
each vertex.  A teardrop (of dimension two) is a manifold with corners but 
not nice. In order to define a good family of torus manifolds, 
we shall introduce some terminology for a compact nice manifold with corners. 

Let $Q$ be a compact nice manifold with corners.  
Faces of $Q$ can naturally be defined and we understand that $Q$ itself 
is a face. We say that $Q$ is a \emph{homology cell} if 
all faces of $Q$ are acyclic, and a \emph{homology polytope} if it is 
a homology cell and any multiple intersection of faces is connected 
whenever it is non-empty. We also say 
that $Q$ is a \emph{homotopy cell} (resp. \emph{homotopy polytope}) if it is 
a homology cell (resp. homology polytope) and all faces are simply connected 
so that all faces are contractible. 
A simple convex polytope is a homotopy polytope and the orbit space 
$S^{2n}/(S^1)^n$ 
in Example~\ref{S2n} is a homotopy cell but not a homotopy polytope when 
$n\ge 2$. 
The face poset of $Q$ 
(with reversed order by inclusion) is a simplicial poset and is 
the face poset of a simplicial complex if any multiple 
intersection of faces of $Q$ is connected whenever it is non-empty.  
The following question asks whether homotopy cells or homotopy polytopes 
can be determined combinatorially. 

\begin{prob} \label{hopo}
Are homotopy cells (or homotopy polytopes) homeomorphic 
as manifold with corners if their face posets are isomorphic?
\end{prob} 

It is shown in \cite{ma-pa06} that if $H^{odd}(M)=0$ for a torus manifold $M$, 
then the torus action on $M$ is locally standard, and 
moreover shown that the orbit space of 
a locally standard torus manifold $M$ is a homology cell (resp. homology 
polytope) if and only if $H^{odd}(M)=0$ (resp. $H^*(M)$ is generated by 
$H^2(M)$ as a ring). However, the orbit space itself or its faces may 
have a non-trivial fundamental group. 
Since a torus group is connected, simply connectedness of a space with 
torus action is inherited to its orbit space 
(\cite[Corollary 6.3 in p.91]{bred72}). 
Therefore, if a torus manifold $M$ satisfies the following 
two conditions: 
\begin{enumerate}
\item[(1)] $H^{odd}(M)=0$ (resp. $H^*(X)$ is generated by $H^2(M)$ as ring),
\item[(2)] $M$, $M_i$'s and connected components of any multiple intersection 
of $M_i$'s are all simply connected,
\end{enumerate}
then the orbits space $M/(S^1)^n$ is a homotopy cell 
(resp. homotopy polytope). 
We believe that torus manifolds satisfying the two conditions above 
will constitute a good family of manifolds for which the toric theory 
can be developed in the topological category in a nice way. 
Toric or quasitoric manifolds are contained in this family.  

As pointed out in Section~\ref{2}, the data of the fan of a toric manifold $X$ 
is equivalent to the data of the simplicial complex $\Sigma_X$ together with 
the set of vectors $\{v_i\}$. 
%
For a torus manifold $M$, we still have the vectors $\{ v_i\}$ and the face 
poset of $M/(S^1)^n$ takes the place of the simplicial complex $\Sigma_X$.  
The same argument as in 
\cite{da-ja91} will show that the homeomorphism type of a torus manifold 
$M$ satisfying the 
two conditions above will be determined by the face poset of $M/(S^1)^n$ 
together with the set of vectors $\{ v_i\}$ if Problem~\ref{hopo} is 
affirmatively answered.  We may ask the problems in 
Sections~\ref{4} and~\ref{5} even for those torus manifolds. 
The reader can find related study in dimension 4 (\cite{or-ra70}) and 
in dimension 6 (\cite{mcga76}).

\end{document}